\documentclass[12pt, a4paper]{article}

\usepackage{graphicx}

\usepackage[T1]{fontenc}
\usepackage{multirow}
\usepackage{float}
\usepackage[centertags]{amsmath}
\usepackage{amsfonts}
\usepackage{amssymb}
\usepackage{amsthm}
\usepackage{authblk}

\DeclareMathOperator{\Gal}{Gal}
\DeclareMathOperator{\Ker}{Ker}
\DeclareMathOperator{\diag}{diag}

\newtheorem{thm}{Theorem}[section]
\newtheorem{cor}[thm]{Corollary}
\newtheorem{lem}[thm]{Lemma}
\newtheorem{prop}[thm]{Proposition}
\theoremstyle{definition}
 \newtheorem{defn}[thm]{Definition}
\theoremstyle{remark}
 \newtheorem{rem}[thm]{Remark}
\theoremstyle{remark}
 \newtheorem{ex}[thm]{Example}
\theoremstyle{definition}
 \newtheorem{conj}[thm]{Conjecture}

\numberwithin{equation}{section}

\begin{document}
\begin{sloppy}

\title{Classification of quantum groups and Belavin--Drinfeld cohomologies for orthogonal and symplectic Lie algebras}
\author[*]{Boris Kadets}
\author[*]{Eugene Karolinsky}
\author[**]{Iulia Pop}
\author[***]{Alexander Stolin}

\affil[*]{\small{Department of Mechanics and Mathematics, Kharkov National University}}
\affil[**]{Department of Mathematical Sciences, Chalmers University of Technology}
\affil[***]{ Department of Mathematics, Gothenburg University, Sweden}
\maketitle
\begin{abstract}
In this paper we continue to study Belavin--Drinfeld cohomology introduced in \cite{SP} and related to the classification of quantum groups whose quasi-classical limit is a given simple complex Lie algebra $\mathfrak{g}$. Here we compute Belavin--Drinfeld cohomology for all non-skewsymmetric $r$-matrices from the Belavin--Drinfeld list for simple Lie algebras of type $B$, $C$, and $D$.


\end{abstract}

\section{Introduction and preliminaries}
In the papers (\cite{EK1}, \cite{EK2}) Kazhdan and Etingof constructed the equivalence between the categories of quantum groups and Lie bialgebras over $\mathbb{C}[[\hslash]]$. The precise statement is in the following theorem.
\begin{thm}\label{KazEti}
Let $\mathrm{Qgroup}$ be the category of quantum groups, i.e.  topologically free cocommutative $\mathrm{mod}\ \hslash$  Hopf algebras over $\mathbb{C}[[\hslash]]$ such that $H/\hslash H$ is a universal enveloping algebra of some Lie algebra $\mathfrak{g}$ over $\mathbb{C}$. Let $\mathrm{LieBialg}$ be the category of topologically free Lie bialgebras over $\mathbb{C}[[\hslash]]$ with $\delta = 0\ \mathrm{mod}\ \hslash$. Then there exists a functor $Lie: \mathrm{Qgroup} \to \mathrm{LieBialg}$ that is an equivalence of categories.
\end{thm}

This theorem can be seen as a quantization of the classical Lie theorem on the equivalence between the categories of simply connected Lie groups and Lie algebras. Lie theorem reduces the problem of classification for (say, semisimple) Lie groups to the same problem for Lie algebras that is much more simple. The same approach works in the quantum group case.

Our goal is to classify quantum groups whose quasiclassical limit is a given semisimple Lie algebra $\mathfrak{g}$. In order to do this we will classify Lie bialgebra structures on $\mathfrak{g}(\mathbb{C}[[\hslash]])$. From Theorem \ref{KazEti} we see that this classification will indeed be a classification of quantum groups.

Any Lie bialgebra structure on $\mathfrak{g}(\mathbb{C}[[\hslash]])$ can be lifted to a Lie bialgebra structure on $\mathfrak{g}(\mathbb{C}((\hslash)))$. Also any Lie bialgebra structure on $\mathfrak{g}(\mathbb{C}((\hslash)))$ defines a Lie bialgebra structure on $\mathfrak{g}(\mathbb{C}[[\hslash]])$  after multiplication by an appropriate power of $\hslash$. So we are left with the problem of classification of Lie bialgebra structures on $\mathfrak{g}(\mathbb{C}((\hslash)))$.

Summarizing the discussion above we have obtained the following: the classification of quantum groups is reduced to the classification of Lie bialgebras over $\mathbb{C}((\hslash))$.

Recall the Belavin-Drinfeld classification:

\begin{thm}[\cite{BD}, \cite{ES}]
Let $\mathfrak{g}$ be a simple Lie algebra over an algebraically closed field of characteristic zero. Then any Lie bialgebra structure on $\mathfrak{g}$ is a coboundary one. Let $r$ be a corresponding $r$-matrix. If $r$ is not skew-symmetric then for some root decomposition we have
$$r = r_0 + \sum _ {\alpha > 0} e_{-\alpha}\otimes e_{\alpha} + \sum_{\alpha \in \mathrm{Span}( \Gamma_1)^{+}}\sum_{k \in \mathbb{N}}e_{-\alpha} \wedge e_{\tau^k(\alpha)} .$$
Here $\Gamma_1, \Gamma_2$ are subsets of the set of simple roots, $\tau: \Gamma_1 \to \Gamma_2$ is an isometric bijection, and for every $\alpha \in \Gamma_1$ there exists $k \in \mathbb{N}$ such that $\tau^k(\alpha) \in \Gamma_1 \setminus \Gamma_2$. The triple $(\Gamma_1, \Gamma_2,\tau)$ is called \emph{admissible}.
The tensor $r_0 \in \mathfrak{h} \otimes \mathfrak{h}$ must satisfy the following two conditions:

(1) $r_0+r_0^{21}=\sum t_k \otimes t_k$, where $t_k$ is an orthonormal basis of $\mathfrak{h}$,

(2) for any $\alpha \in \Gamma_1$ we have $(\tau(\alpha)\otimes \mathrm{id} + \mathrm{id} \otimes \alpha)r_0 = 0 $.
\end{thm}
Our goal is to obtain the version of this theorem over non-closed field.

\subsection{Setting}
In many cases we would prove the results for $B,C$ and $D$ series simultaneously. For this we will introduce some general notation, and then specify it for each case. Note that further in the text we would also make distinction between $D_n$ algebras for odd and even $n$.

By $\mathfrak{g}(F)$ we denote the split simple Lie algebra of rank $n$ over a field $F$ of characteristic zero. We will work with algebras from the $B,C,D$ series. By $G$ we denote a corresponding connected algebraic group, by $H$ a Cartan subgroup of $G$. The simple roots of $\mathfrak{g}$ are denoted $\alpha_1, ..., \alpha_n$. Let $e^{\alpha}$ be the character of $H$ that corresponds to the root $\alpha$. In all cases we will specify the isomorphism of $H$ with the standard torus $(F^*)^n$. Abusing notation, by $(d_1, ..., d_n)$ we denote the element of $H$ that corresponds to an element $(d_1, ..., d_n) \in (F^*)^n $.
We will now specify the choice of matrix representation and root decomposition for the series $B,C,D$. By $M$ we will denote the size of matrices in each matrix representation.\\

{\bf $B_n$ series}

Let $G$ be the subgroup of $SL(2n+1)$ consisting of the matrices satisfying $XBX^T=B$ (we have $M=2n+1$). Here $B$ is the matrix with $1$'s on the anti-diagonal and zeros elsewhere. Then $\mathfrak{g}$ is the algebra of matrices satisfying $XB+BX^T=0$. The Cartan subgroup $H$ (Cartan subalgebra, respectively) can be chosen to be the subset of diagonal matrices of $SL(2n+1)$ (respectively $\mathfrak{sl}(2n+1)$). 
Simple root vectors are given by $e_{\alpha_i}=e_{i, i+1}-e_{2n+2-i, 2n+1-i}$. The Cartan subgroup is the subgroup of diagonal matrices of the form $\diag(d_1, ..., d_n, 1, d_n^{-1}, ..., d_1^{-1})=:(d_1, ..., d_n)$. We have $e^{\alpha_i}(d_1, ..., d_n)=d_id_{i+1}^{-1}$ for $i<n$ and $e^{\alpha_n}(d_1, ..., d_n)=d_n$.\\

{\bf $C_n$ series}

Let $G$ be the subgroup of $SL(2n)$ consisting of the matrices satisfying $XBX^T=B$ (we have $M=2n$). Here $B$ is the matrix with $1$'s on the upper half of the anti-diagonal, minus ones on the lower and zeros elsewhere. E.g. for $n=2$
$$B=\begin{pmatrix}
 0&  0&  0& 1 \\
 0&  0&  1& 0 \\
 0&  -1&  0& 0\\
-1&   0&  0& 0
\end{pmatrix}.$$
 Then $\mathfrak{g}$ is the algebra of matrices satisfying $XB+BX^T=0$. The Cartan subgroup $H$ (Cartan subalgebra, respectively) can be chosen to be the subset of diagonal matrices of $SL(2n)$ (respectively $\mathfrak{sl}(2n)$). 
 Simple root vectors are given by $e_{\alpha_i}=e_{i, i+1}-e_{2n+2-i, 2n+1-i}$ for $i <n$ and $e_{\alpha_n}=e_{n, n+1}-e_{n+1, n}$. The Cartan subgroup is the subgroup of diagonal matrices of the form $\diag(d_1, ..., d_n, d_n^{-1}, ..., d_1^{-1})=:(d_1, ..., d_n)$. We have $e^{\alpha_i}(d_1, ..., d_n)=d_id_{i+1}^{-1}$ for $i<n$ and $e^{\alpha_n}=d_n^2$.\\

{\bf $D_n$ series}

Let $G$ be the subgroup of $SL(2n)$ consisting of the matrices satisfying $XBX^T=B$ (we have $M=2n$). Here $B$ is the matrix with $1$'s on the anti-diagonal and zeros elsewhere. Then $\mathfrak{g}$ is the algebra of matrices satisfying $XB+BX^T=0$. The Cartan subgroup $H$ (Cartan subalgebra, respectively) can be chosen to be the subset of diagonal matrices of $SL(2n)$ (respectively $\mathfrak{sl}(2n)$). 
Simple root vectors are given by $e_{\alpha_i}=e_{i, i+1}-e_{2n+2-i, 2n+1-i}$ for $i <n$ and $e_{\alpha_n}=e_{n-1, n+1}-e_{n+2, n}$. The Cartan subgroup is the subgroup of diagonal matrices of the form $\diag(d_1, ..., d_n, d_n^{-1}, ..., d_1^{-1})=:(d_1, ..., d_n)$. We have $e^{\alpha_i}(d_1, ..., d_n)=d_id_{i+1}^{-1}$ for $i<n$ and $e^{\alpha_n}=d_{n-1}d_n$.

\subsection{Properties of $r$-matrices}
Let $Q: \mathfrak{g} \to \mathfrak{g}^*$ be the natural isomorphism given by Killing form $K$, i.e. $Q(x)(y)=K(x, y)$. Then $\Phi := \mathrm{id} \otimes Q : \mathfrak{g} \otimes \mathfrak{g} \to \mathfrak{g} \otimes \mathfrak{g}^*$ provides an isomorphism between $\mathfrak{g} \otimes \mathfrak{g}$ and $\mathrm{End}(\mathfrak{g})$.

\begin{thm}[\cite{SP}] \label{rstuff}
Let $r_{BD}$ be a non-skewsymmetric r-matrix from the Belavin-Drinfeld list. Denote by $S$ and $N$ the semisimple and nilpotent parts of $\Phi(r_{BD})$. Then $S=\sum_{\alpha>0}e_{-\alpha} \otimes e_{\alpha}+ H$, where $H$ acts by zero on each $\mathfrak{g}_{\alpha}$. The normalizers of the eigenspaces of $S$ with eigenvalues $0$ and $1$ are the Borel subalgebras $\mathfrak{b}_+$ and $\mathfrak{b}_{-}$ respectively. The normalizer of any other eigenspace is the Cartan subalgebra $\mathfrak{h}$.
\end{thm}

\begin{cor}\label{Omega+}
If $r_1, r_2$ are two non-skewsymmetric $r$-matrices, $r_1+r_1^{21}=\alpha \Omega$ and $r_2=r_1 + \beta\Omega$ then $\beta=-\alpha$.
\end{cor}

\begin{proof}
From Theorem \ref{rstuff} we know that normalizers of $0$ and $\alpha$ eigenspaces of $\Phi(r_1)$ are ``big''. Therefore, they must be mapped to the $0$ and $\alpha+2\beta$ eigenspaces of $\Phi(r_2)$ (we have $\Phi(\Omega)=\mathrm{id}$ and $r_2+r_2^{21}=(\alpha+2\beta)\Omega$). On the other hand $\Phi(r_2)=\Phi(r_1)+\beta \mathrm{id}$. The result follows.	
\end{proof}

We will frequently work with the following objects.

\begin{defn}
Let $(\Gamma_1, \Gamma_2, \tau)$ be an admissible triple. Then sets of the form
$\alpha, \tau(\alpha), ..., \tau^{k}(\alpha)$, where $\alpha \in \Gamma_1 \setminus \Gamma_2$ and $\tau^{k}(\alpha) \in \Gamma_2\setminus \Gamma_1$ will be called strings of $\tau$.
\end{defn}

\begin{defn}
The \emph{centralizer} $C(r)$ of an $r$-matrix $r$ is the set of all $X \in G(\overline{F})$ such that $\mathrm{Ad}_Xr=r$.
\end{defn}

The centralizer of a matrix from the Belavin-Drinfeld list can be explicitly described.

\begin{thm} [\cite{SP}] \label{CenterinH}
  For any simple Lie algebra $\mathfrak{g}$ and for any Belavin-Drinfeld matrix $r_{BD}$ we have $C(r_{BD}) \subset H$. If $(\Gamma_1, \Gamma_2, \tau)$ is an admissible triple corresponding to  $r_{BD}$, then $X \in C(r_{BD})$ iff for any root $\alpha \in \Gamma_1 \setminus \Gamma_2$ and for any $k \in \mathbb{N}$ we have $e^{\alpha}(X)=e^{\tau^k(\alpha)}(X)$, i.e. $e^{\alpha}(X)$ is constant on the strings of $\tau$.
\end{thm}
 
\begin{thm}
Let $r$ be a non-skewsymmetric $r$-matrix that defines a Lie bialgebra structure on $\mathfrak{g}(F)$, $r+r^{21}=a\Omega$. Then $a^2 \in F$.
\end{thm}

\begin{proof}
Recall that a Lie bialgebra structure $\delta$ corresponding to $r$ is given by the formula $\delta(a)=[r, a \otimes 1+ 1 \otimes a]$. Fix an arbitrary $\sigma \in \Gal(\overline{F}/F)$. Consider the $r$-matrix $\sigma(r)$.  As $r$ defines a Lie bialgebra structure on $\mathfrak{g}(F)$ we have $\delta(a)=\sigma(\delta(a))$ for all $a \in \mathfrak{g}(F)$. Therefore $\sigma(r)=r+\alpha \Omega$. From Corollary \ref{Omega+} we get $\alpha=0$ or $\alpha=2a$. On the other hand $\sigma(a)\Omega=\sigma(r+r^{21})=\sigma(r)+\sigma(r^{21})=r+r^{21}+2\alpha\Omega=(a+2\alpha) \Omega$. Therefore $\sigma(a)=\pm a$. Thus $\sigma(a^2)=a^2$ for all $\sigma$.
\end{proof}

\subsection{Matrix lemmas}
\begin{lem}\label{blocks}
If for some $X \in GL(n+m, \overline{F})$ and for any $\sigma \in \Gal(\overline{F}/ F)$ we have $X^{-1}\sigma(X)$ is a $n \times n, m \times m$ block matrix, then there exists $Q \in GL(n+m, F)$ and a block matrix $K$ such that $X=QK$.
\end{lem}

\begin{proof}
Let  $X= \left( \begin{array}{cc}
A & B \\
C & D  \\
\end{array} \right)$. Rearranging rows if necessary, we can assume that $A$ and $D$ are non-degenerate (because of Laplace formula).  Then $\sigma(X)=XK$, where $K=\left( \begin{array}{cc}
K_1 & 0 \\
0 & K_2  \\
\end{array} \right) \in GL(n+m, \overline{F})$. Then we have $\sigma(A)=A K_1$ and $\sigma(C)=CK_1$. Since $CA^{-1}$ is $\Gal(\overline{F}/ F)$-stable we have $C=F_1 A$, $F_1 \in GL(m \times n, F)$. Similarly $B=F_2 D$, $F_2 \in GL(n \times m, F)$.  Finally $\left( \begin{array}{cc}
A & B \\
C & D  \\
\end{array} \right) =
\left( \begin{array}{cc}
I & F_2 \\
F_1 & I  \\
\end{array} \right)
\left( \begin{array}{cc}
 A& 0 \\
0 & D  \\
\end{array} \right)$.
\end{proof}

\begin{cor}\label{diagblocks}
If for some $X \in GL(n, F)$ and for any $\sigma \in \Gal(\overline{F}/ F)$ we have $X^{-1}\sigma(X)$ is diagonal, then $X=QD$ where $D$ is diagonal and $Q \in GL(n, F)$.\hfill $\Box$
\end{cor}

\section{Overview of the paper}
Our main goal is to classify Lie bialgebra structures on $\mathfrak{g}(F)$.  We will mostly be interested in the case $F=\mathbb{C}((\hslash))$; however, many general results hold for an arbitrary field of characteristic zero. Because any Lie bialgebra structure on $\mathfrak{g}(F)$ extends to a Lie bialgebra structure on $\mathfrak{g}(\overline{F})$, we have

\begin{thm}
Any Lie bialgebra structure on $\mathfrak{g}(F)$ is a coboundary one given by an $r$-matrix. If $r$ is not skewsymmetric then it has the form $r=\alpha \mathrm{Ad}_X r_{BD}$ where $\alpha \in \overline{F}^*$, $X \in G(\overline{F})$, $r_{BD}$ is an $r$-matrix from the Belavin-Drinfeld list.
\end{thm}

However, not all $r$-matrices of the form $r=\alpha \mathrm{Ad}_X r_{BD}$ define a Lie bialgebra structure on $\mathfrak{g}(F)$.

In what follows we will be interested only in non-skewsymmetric $r$-matrices.
We will classify all (non-skewsymmetic) $r$-matrices that do define a Lie bialgebra structure on $\mathfrak{g}(F)$ up to the following equivalence

\begin{defn}
Two $r$-matrices $r_1$ and $r_2$ are called equivalent if for some $a \in F^*$ and $X \in G(F)$ we have $r_1=a\mathrm{Ad}_X r_2$.
\end{defn}

By the Belavin-Drinfeld classification every $r$-matrix is equivalent to $r_{BD}$ over $\overline{F}$, and $r_{BD}$'s for different triples $(\Gamma_1, \Gamma_2, \tau)$ are not equivalent. However, over $F$ each $\overline{F}$-equivalence class can split into several equivalence classes.
It turns out that if $a \mathrm{Ad}_X r_{BD}$ defines a Lie bialgebra structure on $\mathfrak{g}(F)$ then $a^2 \in F$. Therefore 
we have a class of $r$-matrices for each $a^2 \in F^* / (F^*)^2 $. Elements from different classes are nor equivalent, but each class can further split into several equivalence classes. For $F=\mathbb{C}((\hslash))$ we have only two elements in   $F^* / (F^*)^2 $: the class of $1$ and the class of $\sqrt{\hslash}$. For every element $u \in F^* / (F^*)^2 $ we will introduce the Belavin-Drinfeld cohomologies that parametrize equivalence classes of $r$-matrices of the form $a \mathrm{Ad}_X r_{BD}$, where $a^2=u$ in $F^* / (F^*)^2$. The cohomologies are called twisted if $u \neq 1$ and non-twisted otherwise.

\subsection{Non-twisted cohomologies}
The following theorem holds:

\begin{thm}[\cite{SP}]
$Ad_X r_{BD}$ defines a Lie bialgebra structure on $\mathfrak{g}(F)$ if and only if for any $\sigma \in \Gal(\overline{F}/F)$ we have $X^{-1}\sigma(X) \in C(r)$.
\end{thm}

With that theorem in mind we can define Belavin-Drinfeld cohomologies without referring to Lie bialgebra structures.

\begin{defn}
$X \in G(\overline{F})$ is called a non-twisted Belavin-Drinfeld cocycle for $r_{BD}$ if for any $\sigma \in \Gal(\overline{F}/F)$ we have $X^{-1}\sigma(X) \in C(r_{BD})$. The set of non-twisted cocycles will be denoted $Z(r_{BD})=Z(G, r_{BD})$.
\end{defn}

In other words $X \in Z(r_{BD})$ if $\mathrm{Ad}_X r_{BD}$ defines a Lie bialgebra structure on $\mathfrak{g}(F)$.

\begin{defn}
Two cocycles $X_1, X_2 \in Z(r_{BD})$ are called equivalent if there exist $Q \in G(F)$ and $C \in C(r_{BD})$ such that $X_1=QX_2C$.
\end{defn}

It is easy to see that $X_1$ and $X_2$ are equivalent if and only if $\mathrm{Ad}_{X_1} r_{BD}$ and  $\mathrm{Ad}_{X_2} r_{BD}$ are equivalent as $r$-matrices.

\begin{defn}
The set of equivalence classes of non-twisted cocycles is denoted $H_{BD}^1(r_{BD})=H_{BD}^1(G, r_{BD})$ and is called non-twisted Belavin-Drinfeld cohomologies.
\end{defn}

\begin{rem}
The equivalent way to treat Belavin-Drinfeld cohomology is to understand it as a Galois cohomology of the Galois module $C(r)$ (over the absolute Galois group of $F$). This justifies the use of the word ``cohomology''. Also this approach enables us to put a natural group structure on $H_{BD}^1(r)$. Indeed, for any $X \in Z(r)$ the map $\Gal(\overline{F}/F) \to C(r)$ given by $\sigma \mapsto X^{-1}\sigma(X)$ is a 1-cocycle. If such a map is equivalent to identity then there exists $C \in C(r)$ such that $(XC^{-1})^{-1}\sigma(XC^{-1})=1$ for all $\sigma$. Then one can find $Q \in G(F)$ such that $XC^{-1}=Q$. This means that our equivalence relation is the same as the equivalence of Galois cocycles. We also need to prove that any Galois cocycle is of the given form. We do not know how to prove this in general, however for the algebras of $B,C,D$ types we sketch a proof.

 First one observes that by Hilbert's theorem 90 it is always possible to find $Y \in GL(n, \overline{F})$ such that $Y^{-1}\sigma(Y)=X^{-1}\sigma(X) \in C(r)$ for all $\sigma \in \Gal(\overline{F}/F)$.  With the suitable matrix representation of the group $G$ we can assume that $C(r)$ is a subset of the set of diagonal matrices. Then by Lemma \ref{blocks} we can assume that $Y$ is diagonal. By looking at the equations that define $H$ inside the set of diagonal matrices we can decompose $Y$ as $Y=QY'$ with $Y' \in C(r)$ and $Q \in \diag(n, F)$. Then $Y'$ is in $C(r)$ and gives rise to the same cocycle as $Y$.
\end{rem}
Computation of $H_{BD}^1(G, r_{BD})$ is possible because of the easy description of $C(r_{BD})$ (see Theorem \ref{CenterinH}).

 \subsection{Twisted cohomologies}
Twisted case is somewhat similar to the non-twisted one, however computations become more complicated.
Fix a non-zero $a\in \overline{F}$ such that $a^2 \in F$. We will mostly work with $F=\mathbb{C}((\hslash))$ and $a=\sqrt{\hslash}$. We have the following theorem.
\begin{thm}
$a \mathrm{Ad}_X r_{BD}$ defines a Lie bialgebra structure on $\mathfrak{g}(F)$ if and only if $X$ is a non-twisted cocycle for the field $F[a]$ and $\mathrm{Ad}_{X^{-1}\sigma_0(X)}r_{BD}=r_{BD}^{21}$. Here $\sigma_0$ is the nontrivial element of $\Gal(F[a]/F)$.
 \end{thm}
To deal with the condition  $\mathrm{Ad}_{X^{-1}\sigma_0(X)}r_{BD}=r_{BD}^{21}$ we will classify all triples $(\Gamma_1, \Gamma_2, \tau)$ such that $r_{BD}^{21}$ and $r_{BD}$ are conjugate. In each case we will find suitable $S \in G(F)$ such that $r_{BD}^{21}=\mathrm{Ad}_S r_{BD}$. Then we can define Belavin-Drinfeld cocycles and cohomologies similar to the non-twisted case. In all cases $S^2=\pm 1$.
\begin{defn}
 $X \in G(\overline{F})$ is called the Belavin-Drinfeld twisted cocycle if for any $\sigma \in \Gal(\overline{F}/F[a])$ we have $X^{-1}\sigma(X) \in C(r_{BD})$ and $SX^{-1}\sigma_0(X) \in C(r_{BD})$. The set of Belavin-Drinfeld twisted cocycles is denoted $\overline{Z}(r_{BD})=\overline{Z}(G, r_{BD})$.
\end{defn}
\begin{defn}
 Two twisted cocycles $X_1, X_2$ are called equivalent if there exist $Q \in G(F)$ and $C \in C(r_{BD})$ such that $X_1=QX_2C$. The set of equivalence classes of twisted cocycles is called twisted Belavin-Drinfeld cohomologies and is denoted by $\overline{H}_{BD}^1(r_{BD})=\overline{H}_{BD}^1(G, r_{BD})$.
 \end{defn}
\begin{prop}
Two $r$-matrices $a\mathrm{Ad}_X r_{BD}$ and $a\mathrm{Ad}_Y r_{BD}$ are conjugate by an element of $G(F)$ if and only if $X$ and $Y$ are equivalent as twisted cocycles. Thus the set of equivalence classes of twisted Lie bialgebra structures is $\overline{H}^{1}_{BD}(G,r)$.
\end{prop}
To find $\overline{H}_{BD}^1(r_{BD}, G)$ we proceed as follows.
First we find a matrix $J$ such that $J \in GL(F[a])$ and $\overline{J}=SJ$ (where $\overline{T}:=\sigma_0(T)$). Then we will show that all Belavin-Drinfeld cocycles are of the form $RJD$ with $R \in GL(F)$ and $D$ diagonal. We will then show that for a diagonal matrix $D$ that satisfies some additional assumptions we can find $R$ such that $RJD$ is a cocycle. We will also see that the equivalence of $R_1JD_1$ and $R_2JD_2$ can be understood purely in terms of $D_1$ and $D_2$.  Therefore the problem reduces to classification of classes of diagonal matrices under some equivalence. This will be done explicitly.

\subsection{Results}
Here we summarize the results of the paper.

\begin{table}[H]
\begin{tabular}{|c|c|c|c|}
\hline
Algebra                & Triple type                                                                                                                & \begin{tabular}[c]{@{}c@{}}$H_{BD}^1$ for \\an arbitrary field\end{tabular}  & $H_{BD}^1$ for $\mathbb{C}((\hslash))$ \\ \hline
$B_n$                  &                                                                                                                            & trivial                         &                                      \\ \hline
$C_n$                  &                                                                                                                            & trivial                         &                                      \\ \hline
\multirow{2}{*}{$D_n$} & \begin{tabular}[c]{@{}c@{}}there exists a \\ string of $\tau$ that\\ contains $\alpha_{n-1}$\\ and $\alpha_n$\end{tabular} & $F^*/(F^*)^2$ & 2 elements                           \\ \cline{2-4}
                       & \begin{tabular}[c]{@{}c@{}}$\alpha_{n-1}$ and $\alpha_n$\\ do not belong to \\ the same string of $\tau$\end{tabular}      & trivial                         & \multicolumn{1}{l|}{}                \\ \hline
\end{tabular}
\end{table}

\begin{table}[H]
\begin{tabular}{|c|c|c|c|}
\hline
Algebra                & \multicolumn{2}{c|}{Triple type}                                                                                                                                                                                                                                                                                                                                                                                                          & $\overline{H}_{BD}^1$ for $\mathbb{C}((\hslash))$ \\ \hline
\multirow{2}{*}{$B_n$} & \multicolumn{2}{c|}{Drinfeld-Jimbo}                                                                                                                                                                                                                                                                                                                                                                                                       & one element                                     \\ \cline{2-4}
                       & \multicolumn{2}{c|}{not DJ}                                                                                                                                                                                                                                                                                                                                                                                                               & empty                                           \\ \hline
\multirow{2}{*}{$C_n$} & \multicolumn{2}{c|}{Drinfeld-Jimbo}                                                                                                                                                                                                                                                                                                                                                                                                       & one element                                     \\ \cline{2-4}
                       & \multicolumn{2}{c|}{not DJ}                                                                                                                                                                                                                                                                                                                                                                                                               & empty                                           \\ \hline
\multirow{5}{*}{$D_n$} & \multirow{2}{*}{even $n$}  & Drinfeld-Jimbo                                                                                                                                                                                                                                                                                                                                                                                                & one element                                     \\ \cline{3-4}
                       &                           & not DJ                                                                                                                                                                                                                                                                                                                                                                                                        & empty                                           \\ \cline{2-4}
                       & \multirow{3}{*}{odd $n$} & \begin{tabular}[c]{@{}c@{}}$\Gamma_1=\{\alpha_{n-1}\}$\\ $\tau(\alpha_{n-1})=\alpha_n$;\\ $\Gamma_1=\{\alpha_n\}$\\ $\tau(\alpha_n)=\alpha_{n-1}$;\\ $\Gamma_1=\{\alpha_{n-1}, \alpha_k\}$, $k \neq n$\\ $\tau(\alpha_{n-1})=\alpha_k, \tau(\alpha_k)=\alpha_n$;\\ $\Gamma_1=\{\alpha_{n}, \alpha_k\}$, $k \neq n-1$\\ $\tau(\alpha_{n})=\alpha_k, \tau(\alpha_k)=\alpha_{n-1}$\end{tabular} & two elements                                    \\ \cline{3-4}
                       &                           & Drinfeld-Jimbo                                                                                                                                                                                                                                                                                                                                                                                                & one element                                     \\ \cline{3-4}
                       &                           & other $r_{BD}$                                                                                                                                                                                                                                                                                                                                                                                                        &empty                                                 \\ \hline
\end{tabular}
\end{table}

\section{Non-twisted cohomologies}


\begin{lem}\label{DJtriv}
The Belavin-Drinfeld cohomology associated to the Drinfeld-Jimbo $r$-matrix is trivial.
 \end{lem}
\begin{proof}
 Let $X \in Z(r_{DJ})$. By Theorem \ref{CenterinH} and Lemma \ref{diagblocks}  we have $X=QD$ where $Q \in GL(M, F)$ and $D \in \diag(M, \overline{F})$. Hence, for any $\sigma \in \Gal(\overline{F}/F)$ we have $D^{-1}\sigma(D) \in C(r_{DJ}) \subset H$. This implies $d_id_{M+1-i}=k_i \in F$ (for $B_n$ series we also have $d_{[(M+1)/2]} \in F$). Let $K=\diag(k_1, ..., k_{[M/2]}, 1, 1, ...,1)$ (for $B_n$ let $K=\diag(k_1, ..., k_{[M/2]}, d_{[(M+1)/2]}, 1, ...,1)$). Then $D_1=K^{-1}D \in G$. We have $X=(QK) \cdot D_1$. Thus $X \sim I$.
 \end{proof}

 \begin{rem}
 It is easy to see that $C(r_{DJ})=H$. Therefore for an arbitrary $r$-matrix $r$ we have $Z(r) \subset Z(r_{DJ})$.
 \end{rem}
 Let $r$ be an $r$-matrix from the Belavin-Drinfeld list associated to a triple $(\Gamma_1, \Gamma_2, \tau)$.

\begin{thm}\label{non-tw}
For any $r$ matrix in the $B_n, C_n$ series $H_{BD}^1(r)$ contains one element. For the $D_n$ series if there exists $k \in \mathbb{Z}$ such that $\tau^{k}(n-1)=n$, then $H_{BD}^1(r)$ is parametrized by the elements of $F^*/(F^*)^2$. Otherwise $H_{BD}^1(r)$ is trivial.
 \end{thm}
\begin{proof}
First note that in the $B_n$ and $C_n$ series there is no $k$ such that $\tau^{k}(n-1)=n$, because $\alpha_{n-1}$ and $\alpha_n$ have different length.

Consider a homomorphism $T:H(\overline{F})\to \diag(n, \overline{F})\subset GL(n, \overline{F})$ defined by $T(d_1, ..., d_n)=(s_1, ..., s_{n})$, $s_i=e^{\alpha_i}(d_1,..., d_n)$. Then $T$ is surjective and its kernel is the subset of $(d_1, ...,d_n) \in H$ such that $d_1=...=d_n$. For $s_1, ..., s_{n} \in F$ we can find an element of $H(F)$ in $T^{-1}(s_1, ..., s_{n})$ if and only if the following holds: for the $D_n$ case $s_{n-1}s_n$ is a perfect square, for the $B_n$ case $s_n$ is a perfect square, and always for the $C_n$ series.

For a positive number $i$ such that $\alpha_i \in \Gamma_1$ let us consider the string of $\tau$ that contains $\alpha_i$. By $\eta(i)$ we will denote the smallest number $k$ such that $k>i$ and $\alpha_k$ lies in the same string as $\alpha_i$. If $\alpha_i \not\in \Gamma_1$, then we define $\eta(i)=0$. For $D \in H$ let $s_i=e^{\alpha_i}(D)$.

For any $X \in Z(r_{BD})$  we have a decomposition $X=QD$. Let $D=(d_1, ..., d_n)$ by Lemma \ref{DJtriv}.

{\bf Case 1}. $\eta(n-1)=0$. If $D=(d_1, ..., d_n)$ is a diagonal matrix then the tuple $(s_1, ..., s_{n-1})$  defines $D$ up to a scalar. $D \in C(r_{BD})$ iff $s_i$'s are equal on the strings of $\tau$. We have $D^{-1}\sigma(D) \in C(r_{BD})$. This means that for all $\alpha_i, \alpha_j$ such that $\tau^k(\alpha_i)=\alpha_j$ we have $s_i=k_is_j, k_i \in F$. We can find a matrix $D'=(d_1', ..., d_n') \in H(F)$ for which $s_i'=k_is_j'$ for all $i, j$ such that $\tau^k(\alpha_i)=\alpha_j$ for some $k$.  Then the decomposition $X=(QD') \cdot I \cdot (D'^{-1}D)$ gives an equivalence between $X$ and $I$.

{\bf Case 2.}  $\eta(n-1)=n$ (possible only in the $D_n$ algebra). Again we have $s_i=k_is_{\eta(i)}$ for $k_i \in F$. The equations $s_i=k_is_{\eta(i)}$ for $i=1, ..., n-1$ and $s_{n-1}=k_{n-1}^2s_n$ have a solution $d_1', ..., d_n' \in F$. Let $I'=(1,1, ..., 1, \sqrt{k_{n-1}})$. Then $X=(QD')I'(D'^{-1}D)$. Therefore, every cocycle is equivalent to one of the form $(1, ..., 1, \sqrt{k})$ for some $k \in F$. Let $I'=(1, .., 1, \sqrt{k})$ and $I''=(1, 1, .., 1, \sqrt{l})$. If $I'' \sim I'$, then $I'=QI''C$, where $Q \in H(F)$ and $C \in C(r_{BD})$. But then  $k/l$ must be a square in $F$. Therefore, $H_{BD}^1(r)$ is parametrized by the elements of $F^*/(F^*)^2$.
\end{proof}

\section{Twisted cohomologies}

In this section we work over the field $\mathbb{K}=\mathbb{C}((\hslash))$.

From now on we will make distinction not only between $B,C,D$ series, but also between the algebras from $D_n$ series for odd and even $n$.
By $S$ we denote the following matrix:\\
the matrix with $1$'s on the antidiagonal for the $D_{n}$, $n$ even, and $B_n$ algebras;\\
the matrix with $1$'s on the antidiagonal but with central $2 \times 2$ block replaced with a unit matrix for $D_{n}$ with odd $n$;\\
the matrix with $-1$'s on the upper half of the antidiagonal and $1$'s on the lower half for the $C_n$ algebra.

\begin{ex}
The matrices $S$ for the algebras $D_2, D_3, B_2$ and $C_2$ respectively are\\
$\begin{pmatrix}
 &  &  & 1\\
 &  &1  & \\
 &1  &  & \\
1 &  &  &
\end{pmatrix},
\begin{pmatrix}
 &  &  &  &  & 1\\
 &  &  &  &1  & \\
 &  &1  &0  &  & \\
 &  &0  &1  &  & \\
 &  1&  &  &  & \\
1 &  &  &  &  &
\end{pmatrix},
\begin{pmatrix}
 &  &  &  &1 \\
 &  &  &1  & \\
 &  &1  &  & \\
 &1  &  &  & \\
1 &  &  &  &
\end{pmatrix},
\begin{pmatrix}
 &  &  & -1\\
 &  &-1  & \\
 & 1 &  & \\
1 &  &  &
\end{pmatrix}.
$
\end{ex}
Note that we always have $S \in G$ and $S^2= \pm \mathrm{id}$ with $S^2=-\mathrm{id}$ only for $C_n$ algebra.

Recall that if $\overline{H}^1_{BD}(G,r)$ is nonempty then $r^{21}$ and $r$ are conjugate.

\begin{lem}\label{r-conj}
If $r$ corresponds to a triple $(\Gamma_1, \Gamma_2, \tau)$  and $r$ and $r^{21}$ are conjugate, then $r^{21}$ and $\mathrm{Ad}_S r$ are conjugate by an element of $C(r_{DJ})$.
\end{lem}

\begin{proof}
It is easy to see that $\mathrm{Ad}_Se_{\alpha}=e_{-\alpha}$ for all cases except $D_n$ with odd $n$. For $D_n$ and odd $n$ we have $\mathrm{Ad}_{S}e_{\alpha_i}=e_{-\alpha_i}$ for $i=1,...,n-1$ and $\mathrm{Ad}_{S}e_{\alpha_n}=e_{-\alpha_{n-1}}, \mathrm{Ad}_{S}e_{\alpha_{n-1}}=e_{-\alpha_n}$. Note that $\mathrm{Ad}_S$  acts as $-1$ on $\mathfrak{h} \wedge \mathfrak{h}$. Thus $\mathrm{Ad}_{S}r_{DJ}=r_{DJ}^{21}$. This means that $\mathrm{Ad}_Sr$  and $r^{21}$ have the same semisimple parts. As they are conjugate we can conclude that they are conjugate by an element of $C(r_{DJ})$.
\end{proof}

\begin{lem}\label{r-list}
If $r$ corresponds to a triple $(\Gamma_1, \Gamma_2, \tau)$  and $r$ and $r^{21}$ are conjugate, then the triple is one of the following:\\
for $B_n, C_n$ or $D_n$ with even $n$: $\Gamma_1=\Gamma_2=\emptyset$;\\
for odd $n$:
$\Gamma_1=\Gamma_2=\emptyset$ or one of the following cases:\\
1. $\Gamma_1=\{\alpha_{n-1}\}$, $\tau(\alpha_{n-1})=\alpha_n$;\\
2. $\Gamma_1=\{\alpha_{n}\}$, $\tau(\alpha_{n})=\alpha_{n-1}$;\\
3. $\Gamma_1=\{\alpha_{n-1}, \alpha_k\}$, $\tau(\alpha_{n-1})=\alpha_k, \tau(\alpha_k)=\alpha_{n}$;\\
4. $\Gamma_1=\{\alpha_{n}, \alpha_k\}$, $\tau(\alpha_{n})=\alpha_k, \tau(\alpha_k)=\alpha_{n-1}$.\\
In these cases $r^{21}=\mathrm{Ad}_S r$.
\end{lem}

\begin{proof}
Form Lemma \ref{r-conj} we know that $\mathrm{Ad}_S r$ and $r^{21}$ are conjugate by an element of $C(r_{DJ})$.
Note that $e_\alpha \wedge e_\beta$ is a common eigenvector for $C(r_{DJ})$. If $e_{-\alpha} \wedge e_{\beta}$, $\beta\neq-\alpha$ is a summand of $r$ then $\mathrm{Ad}_Sr$ contains $e_{\alpha} \wedge e_{-\beta}$ and $r^{21}$ contains $e_{-\alpha} \wedge e_{\beta}$. Therefore, in all cases except $D_n$ for odd $n$ if $\alpha=\tau(\beta)$, then $\beta=\tau^{k}(\alpha)$ for some $k$, i.e. $\Gamma_1=\Gamma_2=\emptyset$. For $D_n$ with odd $n$ the same works if $\alpha,\beta \notin \{\alpha_{n-1}, \alpha_n \}$. The only Belavin-Drinfeld triples that do not have such elements are listed in the statement. It is easy to see that for each of them $r^{21}=\mathrm{Ad}_{S}r$.
\end{proof}

\begin{rem}
Assume that $X\in G(\overline{\mathbb{K}})$ is a Belavin--Drinfeld twisted cocycle associated to $r$.  The identity $(\mathrm{Ad}_{X^{-1}\sigma_{2}(X)}\otimes \mathrm{Ad}_{X^{-1}\sigma_{2}(X)})(r)=r^{21}$ is equivalent to
$(\mathrm{Ad}_{S^{-1}X^{-1}\sigma_{2}(X)}\otimes \mathrm{Ad}_{S^{-1}X^{-1}\sigma_{2}(X)})(r)=r$, which implies that $S^{-1}X^{-1}\sigma_{2}(X)\in C(r)$. Thus $X^{-1}\sigma_{2}(X)=SD$ with $D\in C(r)$.

On the other hand, $(\mathrm{Ad}_{X^{-1}\sigma(X)}\otimes \mathrm{Ad}_{X^{-1}\sigma(X)})(r)=r$ holds for any $\sigma\in \Gal(\overline{\mathbb{K}}/\mathbb{K}[\sqrt{\hslash}])$. This in turn is equivalent to $X^{-1}\sigma(X)=D_{\sigma}\in C(r)$.
\end{rem}

We will now introduce a specific matrix $J \in GL(M, \mathbb{K}[\sqrt{\hslash}])$ for the $B,C,D$ series. In every case we will have $\overline{J}:=\sigma_{0}(J)=JS$. Again, notice the distinction between $D_n$ for odd and even $n$.

For $D_n$, $n$ even denote by $J=J(D_n)\in GL(2n,\mathbb{K}[\sqrt{\hslash}])$ the matrix
with the following entries:
$a_{kk}=1$ for $k \leq n$, $a_{kk}=-\sqrt{\hslash}$ for $k \geq n+1$, $a_{k,2n+1-k}=1$
for $k \leq n$, $a_{k,2n+1-k}=\sqrt{\hslash}$ for $k \geq n+1$.

For $D_n$, $n$ odd  let $J$ be a matrix $J(D_{n-1})$ with the unit matrix inserted as the central $2 \times 2$ block.

For $B_n$ let $J$ be a matrix $J(D_n)$ with the unit matrix inserted as the central
$1 \times 1$  block.

For $C_n$ take $J$ to be $J(D_n)$.

 E.g.\ for $D_2$ and $C_2$ we have $$J=\left(\begin{array}{cccc}
 1 & 0 &0 & 1\\
0 & 1 & 1 & 0\\
0 &\sqrt{\hslash} &-\sqrt{\hslash} &0\\
\sqrt{\hslash} & 0 & 0 & -\sqrt{\hslash} \end{array}\right),$$ for $D_3$ we have
$$J= \left(\begin{array}{cccccc}
 1 & & &... & & 1\\
  &  1& &... &1 &\\
... & & 1 & 0 & ... &\\
... & & 0 & 1 & ... &\\
& \sqrt{\hslash} & &... & - \sqrt{\hslash} &\\
\sqrt{\hslash} &  & &... & & -\sqrt{\hslash} \end{array}\right),$$ and for $B_2$ we have
$$J=\left(\begin{array}{ccccc}
 1 & 0 & 0 &0 & 1\\
0 & 1 &0 & 1 & 0\\
0 &0 &1 &0 & 0\\
0 &\sqrt{\hslash} & 0 &-\sqrt{\hslash} &0\\
\sqrt{\hslash} & 0 &0 &  0 & -\sqrt{\hslash} \end{array}\right).$$

\begin{lem}\label{o_RJD}
 Let $X\in \overline{Z}(G,r)$. Then there exist $R\in GL(M,\mathbb{K})$
and $D\in \mathrm{diag}(M,\overline{\mathbb{K}})$ such that $X=RJD$. Moreover,
$D=\mathrm{diag}(d_{1},..., d_{M})$, where $d_{i}d_{M-i}\in \mathbb{K}$
for all $i$ (for $B_n$ series we also have $d_{n+1}=1$, for $C_n$ series we have $d_id_{M-i} \in \sqrt{\hslash} \mathbb{K}$), and for any $\sigma \in \Gal(\overline{\mathbb{K}}/\mathbb{K}[\sqrt{\hslash}])$ we have $D^{-1}\sigma(D) \in C(r)$.
\end{lem}

\begin{proof}
By Theorem \ref{CenterinH} we have $C(r)\subset H$.
Note that in the equation $\overline{X}=XSD$ the matrix $SD$ is a block matrix with $2 \times 2$ blocks. Applying Lemma \ref{blocks} we see that $X=RX_1$ with $X_1$ having the same block form. If we would provide an appropriate decomposition for each $2 \times 2$ block we would obtain the desired decomposition.
 In each block the equation can be equivalent to the non-twisted cohomology equation (if in $S$ we have a unit block) or as one of the equations considered below. For $\sigma \in \Gal(\overline{\mathbb{K}}/\mathbb{K}[\sqrt{\hslash}])$ we have $RJDC=\sigma(RJD)=RJ\sigma(D)$, where $C \in C(r)$.

Let $X=\left(\begin{array}{cc} a & b \\
c & d \end{array}\right)\in GL(2,\mathbb{K}[\sqrt{\hslash}])$
satisfy $\overline{X}=XWD$ with $D=\mathrm{diag}(d_{1},d_{2})\in SO(2,\mathbb{K}[\sqrt{\hslash}])$, where $W=\begin{pmatrix}
0 & 1\\
1 & 0
\end{pmatrix}$. The identity is equivalent to the following system: $\overline{a}=bd_{1}$,  $\overline{b}=ad_{2}$, $\overline{c}=dd_{1}$,
$\overline{d}=cd_{2}$. Note that $cd\neq 0$ (otherwise both $c$ and $d$ should be zero). Let $a/c=a'+b'\sqrt{\hslash}$. Then $b/d=a'-b'\sqrt{\hslash}$. One can immediately check that $X=RJ'D'$, where
$R=\left(\begin{array}{cc} a' & b' \\
1 &0 \end{array}\right)\in GL(2,\mathbb{K})$,
 $J'= \left(\begin{array}{cc} 1 & 1 \\
\sqrt{\hslash} &-\sqrt{\hslash} \end{array}\right)$, $D'=\mathrm{diag}(c,d)\in \mathrm{diag}(2,\mathbb{K}[\sqrt{\hslash}])$. Note that $\overline{cd}=cd d_1d_2=cd$, thus $cd \in \mathbb{K}$.

Let $X=\left(\begin{array}{cc} a & b \\
c & d \end{array}\right)\in GL(2,\mathbb{K}[\sqrt{\hslash}])$
satisfy $\overline{X}=XWD$ with $D=\mathrm{diag}(d_{1},d_{2})\in SO(2,\mathbb{K}[\sqrt{\hslash}])$, where $W=\begin{pmatrix}
0 & -1\\
1 & 0
\end{pmatrix}$. The identity is equivalent to the following system: $\overline{a}=bd_{1}$,  $\overline{b}=-ad_{2}$, $\overline{c}=dd_{1}$,
$\overline{d}=-cd_{2}$. Note that $cd\neq 0$ (otherwise both $c$ and $d$ should be zero). Let $a/c=a'+b'\sqrt{\hslash}$. Then $b/d=a'-b'\sqrt{\hslash}$. One can immediately check that $X=RJ'D'$, where
$R=\left(\begin{array}{cc} a' & b' \\
1 &0 \end{array}\right)\in GL(2,\mathbb{K})$,
 $J'= \left(\begin{array}{cc} 1 & 1 \\
\sqrt{\hslash} &-\sqrt{\hslash} \end{array}\right)$, $D'=\mathrm{diag}(c,d)\in \mathrm{diag}(2,\mathbb{K}[\sqrt{\hslash}])$. Note that $\overline{cd}=-cd d_1d_2=-cd$, thus $cd \in \sqrt{\hslash} \mathbb{K}$.
\end{proof}

\begin{lem}
Let $D \in \mathrm{diag}(M, \overline{\mathbb{K}})$ satisfy $ d_i d_{M+1-i} \in \mathbb{K}$ for all $i$ (and $d_{n+1}=1$ for $B_n$ series; for $C_n$ we have $ d_i d_{M+1-i} \in \sqrt{\hslash} \mathbb{K}$ instead). Then there exists $R \in GL(M, \mathbb{K})$ such that $RJD \in G$.
\end{lem}

\begin{proof}
Let $X=RJD \in G$. Then $X$ must satisfy $X^{T}SX=S$. Denote $Y:=R^{-1}$. Then
$YSY^{T}=JDS(JD)^{T}$. Therefore, we can find the desired $R$ if and only if the quadratic forms with matrices $S$ and $JDS(JD)^T$ are equivalent over $\mathbb{K}$. Note that symplectic forms are always equivalent over a field. Henceforth we consider the cases of $B_n$ and $D_n$ series. We have $$JDS(JD)^T=\mathrm{diag}(2d_{1}d_{2n},2d_{2}d_{2n-1},...,2d_{n}d_{n+1},-2\hslash d_{n}d_{n+1},...,-2\hslash d_{1}d_{2n})$$ for $D_n$ and even $n$,
$$JDS(JD)^T= \mathrm{diag} (2d_{1}d_{2n},2d_{2}d_{2n-1},..., d_nd_{n+1}, d_{n}d_{n+1}, ...,-2\hslash d_{1}d_{2n})$$ for $D_n$ and odd $n$, and  $$JDS(JD)^T=\mathrm{diag}(2d_{1}d_{2n},2d_{2}d_{2n-1},...,2d_{n}d_{n+1}, 1, -2\hslash d_{n}d_{n+1},...,-2\hslash d_{1}d_{2n})$$ for the $B_n$ algebra. Thus it suffices to show that for $d \in \mathbb{K}$ the quadratic forms with matrices $\left(\begin{array} {cccc}
0 & 0 & 0 & 1 \\
0 & 0 & 1 & 0 \\
0 & 1 & 0 & 0 \\
1 & 0 & 0 & 0 \\  \end{array}\right)$,
$\left(\begin{array} {cccc}
d &0 & 0 & 0\\
0 & d & 0 & 0\\
 0 &0 & -\hslash d & 0\\
 0 &0 & 0 & -\hslash d \end{array}\right)$
 and $\left(\begin{array} {cccc}
d &0 & 0 & 0 \\
 0 & d & 0 & 0 \\
 0 & 0 & d & 0 \\
  0 & 0 & 0 & d \\
   \end{array}\right)$ are equivalent. This is straightforward.
\end{proof}

For  $D \in \diag(M, \overline{\mathbb{K}})$ define  $T(D)=SD^{-1}S\overline{D}$ (for the $C_n$ series we should take $T(D)=-SD^{-1}S'\overline{D}$ where $S'$ is the matrix with $1$'s on the antidiagonal, i.e. we ``forget'' the signs).

We will denote by $\mathbf{Z}$ the set of all diagonal matrices $D=\diag(d_1, ..., d_{M}) \in \diag(M, \overline{\mathbb{K}})$ satisfying $d_id_{M+1-i} \in \mathbb{K}$ (also $d_{n+1}=1$ for the $B_n$ algebra, $d_id_{2n+1-i} \in \sqrt{\hslash} \mathbb{K}$ for the $C_n$ algebra) and $D^{-1}\sigma(D) \in C(r)$ for all $\sigma \in \Gal(\overline{\mathbb{K}}/\mathbb{K}[\sqrt{\hslash}])$.

\begin{lem}
Let $X=RJD \in G$, $D \in \mathbf{Z}$. Then $X$ is a twisted cocycle if and only if $T(D) \in C(r)$.
\end{lem}

\begin{proof}
We have $\overline{X}=RJS\overline{D}=XST(D)$.
\end{proof}

\begin{thm}
We have
$$\overline{H}_{BD}^1(G,r)=\frac{T^{-1}(C(r)) \bigcap \mathbf{Z}}{C(r)\Ker T}$$
\end{thm}

\begin{proof}
For an arbitrary $D \in T^{-1}(C(r)) \bigcap \mathbf{Z}$ we can find $R \in GL(M, \mathbb{K})$ such that $X=RJD$ is a cocycle.

Assume that $X=R_1JD_1 \sim Y=R_2JD_2$. Then $Y=QXC$, $Q=R_2JD_2C^{-1}D_1J^{-1}R_1^{-1}$. Since $Q=\overline{Q}$, we have $D_2C^{-1}D_1^{-1}=S\overline{D_2C^{-1}D_1^{-1}}S$ (for $C_n$ series one should take $S'$ instead of $S$ in the last formula). Thus $D_2D_1^{-1} \in C(r)\Ker T$. If $D_2D_1^{-1} \in C(r)\Ker T$, then $D_2D_1^{-1}=D_0C$. If we denote $Q=R_2JD_0J^{-1}R_1^{-1}$, then we have $Y=QXC$.
\end{proof}

\begin{thm}
Let $r$ be from the list of Lemma \ref{r-list}. Then $|\overline{H}_{BD}^1 (G, r)|=1$ for $r=r_{DJ}$ and $|\overline{H}_{BD}^1 (G, r)|=2$ for $r \neq r_{DJ}$.
\end{thm}

\begin{proof}
We will first consider the case of the Drinfeld-Jimbo $r$-matrix. Let $D=\diag(d_1, ..., d_{M})$. Then for $D_n$ with even $n$ we have $$T(D)=\diag(d_{2n}^{-1}\overline{d_1}, d_{2n-1}^{-1}\overline{d_2}, ..., d_1^{-1}\overline{d_{2n}}),$$ for $D_n$ with odd $n$ we have $$T(D)=(d_{2n}^{-1}\overline{d_1}, ..., d_{n}^{-1}\overline{d_n}, d_{n+1}^{-1}\overline{d_{n+1}}, ..., d_{1}^{-1}\overline{d_{2n}}),$$ for  $B_n$  we have $$T(D)=\diag(d_{2n}^{-1}\overline{d_1}, d_{M-1}^{-1}\overline{d_2}, ..., d_{n+1}^{-1}d_n, 1, d_n^{-1}d_{n+1}, ...,  d_1^{-1}\overline{d_{2n}}),$$ anf for the $C_n$ algebra $$ T(D)=(d_{2n}^{-1}\overline{d_1}, ..., d_{n+1}^{-1}\overline{d_n}, -d_{n-1}^{-1}\overline{d_{n+1}}, ..., -d_{1}^{-1}\overline{d_1}) .$$ If $D \in \mathbf{Z}$, then $T(D) \in G$, thus $T(D) \in C(r)$. It is easy to see that an arbitrary matrix from $\mathbf{Z}$ can be decomposed as a product of a matrix from $G$ and a matrix from $\Ker T$. Therefore $|H_{BD}^1(G, r_{DJ})|=1$. For the $D_n$ series decomposing $D \in \mathbf{Z}$ as $D=WD_{1}$ with $W \in \Ker T$ and $D_1 \in G$ one gets the following: if $D\in \mathbb{K}$ (resp.\ $\mathbb{K}[\sqrt{\hslash}]$), then $W, D_1 \in \mathbb{K}$ (resp.\ $\mathbb{K}[\sqrt{\hslash}]$).

Now let $r \neq r_{DJ}$ (this means that we deal with the case $D_n$ for odd $n$). Let $D \in \mathbf{Z}$. Then $D=KF$, where $K \in \diag(2n, \mathbb{K}) \bigcap \Ker T$, $F \in G$. Then $F$ is a non-twisted cocycle for $r$ over $\mathbb{K}[\sqrt{\hslash}]$. Hence, because of factorization by $C(r)$ we can assume that either $F \in \diag(2n,\mathbb{K}[\sqrt{\hslash}])$ or $F=PF_1$, where $F_1 \in \diag(2n,K[\sqrt{\hslash}])$, $P=\diag(1,1, ..., \sqrt[4]{\hslash}, 1/\sqrt[4]{\hslash}, 1, ..., 1)$ (see Theorem \ref{non-tw}). But it is easy to see that $P \in \Ker T$. Therefore in what follows we assume that $D \in GL(2n, K[\sqrt{\hslash}])$, $d_id_{2n+1-i} \in \mathbb{K}$. Finally $$ \overline{H}_{BD}^1(SO(2n),r)=\frac{T^{-1}(C(r)) \bigcap \mathbf{Z}}{C(r)\Ker T}=\frac{A}{C(r)\Ker T},$$
where $A=T^{-1}(C(r))\bigcap O(2n, \mathbb{K}[\sqrt{\hslash}])$.

Recall that we define $(d_1, ..., d_n)=\diag(d_1, ..., d_n, 1/d_n, ..., 1/d_1)$. We will write the defining relations for $A$ and $C(r)$ in each case. Defining relations for $\Ker T$ are $d_1\overline{d_1}=1, ...,d_{n-1}\overline{d_n}=1, \overline{d_n}=d_n$.

\noindent{\bf Case 1.}\\
$A:$ $(d_{n}^{-1}\overline{d_n})^2=1$.\\
$C(r):$ $d_n^2=1$.\\
If $d_n \in \mathbb{K}$, then $(d_1, ..., d_n)=(d_1, ..., d_{n-1}, 1)(1, 1, ..., d_n)\in C(r)\Ker T$. \\
If $d_n \in \sqrt{\hslash} \mathbb{K}$, then $(d_1, ..., d_n)=(1, ..., \sqrt{\hslash})(1, ..., 1, d_n/\sqrt{\hslash})(d_1, ..., d_{n-1},1)$. We have $(1, ..., 1)\not\sim (1, ..., \sqrt{\hslash})$ because if $(a_1, ..., a_n) \in \Ker T \cup C(r)$ then  $a_n \in \mathbb{K}$.
Hence $H_{BD}^1$ contains two elements.

Case 2 is similar to Case 1.

\noindent{\bf Case 3.1} (Case 3 with $k=n-2$).\\
$A: (d_n^{-1}\overline{d_n})^2=1;\ (d_{n-1}\overline{d_{n-1}})^2d_n^{-1}\overline{d_n}=d_{n-2}\overline{d_{n-2}}$.\\
$C(r): d_n^2=1, d_{n-1}^2d_n=d_{n-2}$.\\
If $d_n \in \mathbb{K}$, then
$$(d_1, ..., d_n)=(d_1, ..., d_{n-3}, d_{n-1}^2, d_{n-1},1)(1, ..., 1, d_{n-2}/d_{n-1}^2, 1, d_n).$$
If $d_n \in \sqrt{\hslash}\mathbb{K}$, then
$$(d_1, ..., d_n)=(1, ..., \sqrt{\hslash})(d_1, ..., d_{n-1}^2, d_{n-1},1)(1, ..., d_{n-2}/d_{n-1}^2, 1, d_n/\sqrt{\hslash}).$$

\noindent{\bf Case 3.2} (Case 3 with $k\neq n-2$).\\
 To simplify the notation we assume $k=n-3$ (general case is completely similar)\\
 $C(r): d_n^2=1; d_{n-3}=d_{n-2}d_{n-1}d_n$.\\
 $A: (d_n^{-1}\overline{d_n})^2=1; d_{n-3}\overline{d_{n-3}}=d_{n-2}d_{n-1}d_n^{-1} \overline{d_{n-2}d_{n-1}d_n}$.  \\
 If $d_n \in \mathbb{K}$, then
 $$(d_1, ..., d_n)=(d_1, ..., d_{n-2}d_{n-1}, d_{n-2}, d_{n-1}, 1) (1, ..., d_{n-3}/d_{n-2}d_{n-1}, 1, 1, d_n). $$
 If $d_n \in \sqrt{\hslash}\mathbb{K}$, then
  $$(d_1, ..., d_n)=(1,..., \sqrt{\hslash})(d_1, ..., d_{n-2}d_{n-1}, d_{n-2}, d_{n-1}, 1) (1, ..., d_{n-3}/d_{n-2}d_{n-1}, 1, 1, d_n/\sqrt{\hslash}). $$

Case 4 is similar to Case 3.
\end{proof}

\section{Concluding Remarks}
Consider the isomorphosm $\mathfrak{so}(3,3) \cong \mathfrak{sl}(4)$. Under this isomorphism $F$-points are mapped to $F$-points. We know that $|\overline{H}_{BD}^1(r_{BD}, \mathfrak{sl}(4), GL(4))|=1$ (see \cite{SP}). Our results imply that $\overline{H}_{BD}^1(r_{BD}, \mathfrak{so}(3,3))$ contains $2$ elements. This means that we have two twisted Lie bialgebra structure on $\mathfrak{so}(3,3)$ that are non-conjugate but isomorphic. This motivates the following conjecture:

\begin{conj}
If $r_{BD}$ is an $r$-matrix from the list of  Lemma \ref{r-list} then all twisted Lie bialgebra structures on $\mathfrak{g}$ corresponding to $r_{BD}$ are isomorphic.
\end{conj}

Consider the problem of quantizing Belavin-Drinfeld $r$-matrices. In \cite{ESS} Etingof, Schedler and Schiffman explicitely quantized the whole Belavin-Drinfeld list. Let us see how this is related to our considerations.
Note that any $r$-matrix of twisted type satisfies $r+r^{21}=\sqrt{\hslash} \Omega$, where $\Omega$ is a Casimir element. This means that the quasi-classical limit of the twisted type $r$-matrix is skew. Therefore, Belavin-Drinfeld non-twistes cohomologies parametrize the set of quantizations of a given Belavin-Drinfeld matrix (up to some equivalence that is weaker then isomorphism). Thus, in most cases considered (except for special BD-triples for $D_n$) the Etingof-Schedler-Schiffman quantization is unique.

\medskip

\textbf{Acknowledgment.} The authors are grateful to J.-H. Lu for valuable suggestions.

\end{sloppy}
\end{document}